\documentclass[lettersize,journal,twocolumn]{article}
\usepackage{amsmath,amsfonts}
\usepackage{amsthm}
\usepackage{algorithmic}
\usepackage{algorithm}
\usepackage{array}
\usepackage[caption=false,font=normalsize,labelfont=sf,textfont=sf]{subfig}
\usepackage{textcomp}
\usepackage{stfloats}
\usepackage{url}
\usepackage{verbatim}
\usepackage{graphicx}
\usepackage{cite}
\usepackage{comment}
\usepackage{gensymb}
\usepackage{geometry}
\newcommand{\C}{\mathbb C}
\newcommand{\R}{\mathbb R}
\newcommand{\F}{\mathcal F}
\newcommand{\M}{M}
\newcommand{\T}{^\mathsf{T}}
\newcommand{\E}{\mathbb E}

\newcommand{\inv}{^{-1}}

\newcommand{\trace}{\mathrm{trace}}
\newcommand{\Rtrace}{\mathrm{trace}_\Re}

\newcommand{\rmd}{\mathrm{d}}

\newtheorem{theorem}{Theorem}

\newtheorem{lemma}{Lemma}
\newtheorem{prop}{Proposition}

\begin{document}

\title{Fixed Point Iterations for SURE-based PSF Estimation for Image Deconvolution}

\author{Toby Sanders}
\date{}% <-this % stops a space
% \thanks{Manuscript received ; revised }}


% \IEEEpubid{}

\maketitle

\begin{abstract}
	Stein's unbiased risk estimator (SURE) has been shown to be an effective metric for determining optimal parameters for many applications. The topic of this article is focused on the use of SURE for determining parameters for blind deconvolution. The parameters include those that define the shape of the point spread function (PSF), as well as regularization parameters in the deconvolution formulas. Within this context, the optimal parameters are typically determined via a brute for search over the feasible parameter space. When multiple parameters are involved, this parameter search is prohibitively costly due to the curse of dimensionality. In this work, novel fixed point iterations are proposed for optimizing these parameters, which allows for rapid estimation of a relatively large number of parameters. We demonstrate that with some mild tuning of the optimization parameters, these fixed point methods typically converge to the ideal PSF parameters in relatively few iterations, e.g. 50-100, with each iteration requiring very low computational cost.
\end{abstract}

\section{Introduction}
Image deconvolution is a highly ill-conditioned inverse problem in which an image has been degraded by convolution with a point-spread function (PSF). The inverse problem involves some \emph{undoing} of the convolution via a regularized deconvolution algorithm. When the PSF is known exactly, a large variety of algorithms exist to handle this problem. However, for many deconvolution problems, the PSF is unknown and therefore must also be estimated, either prior to or in conjunction with the deconvolution algorithm. When the PSF is unknown the problem is known as \emph{blind} deconvolution.

The general form of the blind deconvolution problem in this article is to recover an image $u_0 \in \R^{N}$ from its blurry version, $b$, given by
\begin{equation}\label{eq: convModel}
	b = h_0*u_0 + \epsilon,
\end{equation}
where $*$ denotes 2D convolution, $h_0$ is the (unknown) PSF, and $\epsilon$ is inherent noise. To begin to solve this problem, some prior assumptions and/or constraints must be put to action. Some common models implement prior assumptions about the structure of the image and PSF by using regularization norms \cite{fergus2006removing,cho2009fast}. % For example, the total variation norm may be used on the image while an $\ell_2$ norm Tikhonov prior is used for the PSF. One such model may be given by
% \begin{equation}
% \min_{h,u} \frac{1}{2} \| h*u - b\|_2^2 + \lambda_u \| \nabla u \|_1 + \frac{\lambda_h}{2} \| h \|_2^2.
% \end{equation}
Such models are usually non-convex, but pragmatic estimates can be obtained by alternating minimization over estimates of $h_0$ and $u_0$. Throughout this article we will refer to these estimates as $h$ and $u$, respectively. Other priors about the PSF may be contained in the imaging acquisition domain. For example, in telescope imaging and light microscopy, a parametric model for the PSF is designed based on the model of the lense aperture \cite{holmes1992blind, paxman1992joint, sanders2021real}. Some parameters in the model for the PSF are left as free variables, and an alternating optimization approach is used again by alternatively optimizing over the image and PSF.% This problem is still generally non-convex.

This article focuses on the problem of estimating the PSF, whenever an imaging modality-based model for the PSF is essentially unknown. In other words, the source of the blur may be considered arbitrary. This work follows on the effective  approach proposed by Xue and Blu \cite{xue2014novel}, where a parametric model is used for the PSF, e.g. a symmetric Gaussian PSF with unknown variance. The parameter is found by minimizing the Stein's unbiased risk estimator (SURE), which is an estimator for the \emph{blurred} square error. Hence the optimal parameter is determined by minimizing SURE as the objective function. In this sense, once the optimal PSF is determined, any more sophisticated non-blind deconvolution algorithm may be put to use to obtain the final deconvolved image. Hence, for the purposes of this article, the problem of recovering an \emph{ideal} deconvolved image $u$ is considered a separate exercise performed posterior to estimating the PSF, which the reader may refer to the extensive literature on (see e.g. \cite{dabov2008image, romano2017little, zhang2017learning}). 

In the work of Xue and Blu \cite{xue2014novel}, the optimal parameters are found essentially by brute force search over the parameter space, with some mildly more involved methods used when searching for the optimal PSF variance and regularization parameter. Nevertheless, performing a manual search over the parameter space means that this approach is prohibitively limited to the number of parameters that can be used, due to the curse of dimensionality. To this end, this article proposes novel iterative methods for minimizing SURE over the parameter space, namely through fixed point iterations derived from the SURE objective function. We demonstrate that these fixed points generally converge in 50-100 iterations with some mild tuning of the optimization parameters. Moreover, the iterations are extremely lightweight and fast. Using this approach, it is demonstrated that up to 4 free parameters may be solved for simultaneously, while more are likely possible. These parameters include the values for the PSF variances, a rotational parameter $\theta$ for the PSF orientation, and an optimal regularization parameter. 

Provided in the next section is an outline of the technical background material needed for the approach. Namely, SURE is clearly defined, an important result from Xue and Blu is revisited, regularization operators are briefly discussed, and finally a general outline of the computational practices used throughout the article is described. Section 3 provides the main mathematical results in this article and the resulting fixed point equations for the PSF parameters. Section 4 describes in detail several useful parametric models for PSFs. Finally, Section 5 presents the numerical results and the mild empirical optimization tuning that was implemented.

\section{Preliminaries}  
We assume $\epsilon$ in (\ref{eq: convModel}) is i.i.d. mean zero Gaussian with variance $\sigma^2$, though the methodology still applies so long as the covariance matrix of $\epsilon$ is known \cite{eldar2008generalized}. It is often convenient to write the convolution with a PSF $h$ in operator form as
\begin{equation}
	h*u = Hu,
\end{equation}
where $H \in R^{N \times N}$, and by the Fourier convolution theorem
\begin{equation}
	H = \F^{-1} \hat H \F .
\end{equation}
The operator $\F$ is the unitary Fourier transform operator.
The matrix $\hat H$ is diagonal and contains the Fourier transform values of $h$.

SURE provides a statistical estimate of the squared error between an estimate of the blurred solution $H u$ and the true noise free blurred image, $H_0u_0$ \cite{sanders2019notes,tibshirani2015stein}, which we write as
\begin{equation}
	\text{SURE}(u) := \E \| H_0 u_0  - H u \|_2^2.
\end{equation}
The statistical estimator is given by
\begin{equation}\label{eq: SURE}
	\text{SURE}(u) = -N\sigma^2 + \| H u - b\|_2^2 + {2\sigma^2} \sum_{j=1}^{N } \frac{\partial ( H u )_j}{\partial b_j }.
\end{equation}
To make the estimator computationally tractable, it is useful to implement linear Wiener filters for the inverse maps \cite{wiener1964extrapolation}, in which case (\ref{eq: SURE}) can be evaluated rapidly as shown below. The Wiener filter solution is given by \cite{wiener1964extrapolation,xue2014novel}
\begin{equation}\label{eq: Wie}
	u = W_H b = (H\T H + \sigma^2 V )^{-1} H\T b,
\end{equation} 
where $V = \E[ U\T U]^{-1} $. Here $U$ is the circulant matrix formed by putting circle shifted copies of $u_0$ into each column. Namely,
\begin{equation}
U_{ij} = (u_0)_{(i+j-1)mod N}.
\end{equation}
It is easy to see that 
$$
\F^{-1} U \F ,
$$
is a diagonal matrix containing the Fourier transform values of $u_0$. Therefore
$$
\F^{-1} V \F = \F^{-1} [U\T U]^{-1} \F = [ \F^{-1} U\T U \F ]^{-1},
$$
which we can observe is a diagonal matrix containing the inverse of the squared magnitude of the Fourier transform values of $u_0$.

For notational convenience in the proceeding work we define the symmetric positive definite matrix $M$ as
\begin{equation}\label{eq: SPDM}
	\M = (H\T H + \sigma^2 V)^{-1}.
\end{equation} 
The covariance matrix $V$ is not generally known in practice and is instead empirically designed (see below in Section \ref{sec: Reg}). Finally, given this linear Wiener filter solution, SURE in (\ref{eq: SURE}) simplifies to 
{\small{
\begin{equation}\label{eq: SURE-lin}
	\text{SURE}(u)
	  = -N \sigma^2 + \| H u - b\|_2^2 + {2\sigma^2} \trace (H M H\T ).	
\end{equation}
}}
It is very important to note the terms appearing in (\ref{eq: SURE-lin}), as well as similar expressions appearing later in this article, are very simple and fast to implement numerically. This is discussed in more detail in Section \ref{sec: comp}.

An important result from the work in \cite{xue2014novel} to justify minimizing the SURE criterion in (\ref{eq: SURE}) over the PSF parameters is given in the following.
\begin{theorem}\label{thmBlu}
	Consider only linear processings of the form $u = W_H b$ where $W_H$ is defined in (\ref{eq: Wie}), and suppose the true PSF operator is $H_0$. Then minimizing SURE in (\ref{eq: SURE}) with respect to $H$ over all possible convolution operators obtains a solution $H$ satisfying $H H\T  = H_0 H_0\T $.
\end{theorem}
It is important to note that Theorem \ref{thmBlu} depends on knowing the ideal regularization matrix $V$, and we have empirically observed that setting $V = \lambda I$ obtains poor results. However, whenever $\sigma^2 V$ is reasonably approximated (see section below), we have found the results to be generally encouraging. In \cite{xue2014novel}, they use the approximation $\E | (\F u)_k |^2 \propto  1/k^2$, where $k$ is the wave number. In what follows below it is shown that our empirical Tikhonov regularization is quite similar, while having a more intuitive real domain interpretation.

\subsection{Regularization Operators}\label{sec: Reg}
As noted, the matrix $V$ appearing in (\ref{eq: Wie}) is generally unknown, and instead empirical operators are used. For our work we use the surrogate matrix that naturally arises out of the Tikhonov regularization model for image restoration, which is given by
\begin{equation}\label{eq: Tik}
	\begin{split}
		u &= \arg \min_v \| H v - b \|_2^2 + \lambda \| T v \|_2^2\\
		& = (H\T H + \lambda T\T T)^{-1} H\T b.
	\end{split}
\end{equation}
Here, $T$ is the regularization operator often set as a first or second order finite difference operator\cite{sanders2020effective,sanders2018multiscale}. For example in 1D, the first order finite difference operators is given as
\begin{equation}\label{eq: FD}
	T = \left[
	\begin{array}{ccccc}
		-1 &  1 & 0 & \dots & 0\\
		0 & -1 & 1 & \dots & 0\\
		\vdots & \vdots & \ddots  & &  \vdots \\
		0 & \dots & & -1 & 1 \\
		1 & \dots & & 0 & -1
	\end{array}\right],
\end{equation}
% 	\quad \text{and} \quad 
%	T = \left[
%	\begin{array}{ccccc}
%		1 &  -2 & 1 & \dots & 0\\
%		0 & 1 & -2 & \dots & 0\\
%		\vdots & \vdots & \ddots  & &  \vdots \\
%		1 & \dots & & 1 & -2 \\
%		-2 & 1 & \dots & 0 & 1
%	\end{array}\right] .
%\end{equation}
and higher order operators may be obtained by repeatedly applying the first order difference.
Observe that the solution to (\ref{eq: Tik}) is the same as that in (\ref{eq: Wie}) by setting $\sigma^2 V = \lambda T\T T$. For this reason we will typically write $\lambda T\T T$ in place of $\sigma^2 V$, hence $M$ appearing in (\ref{eq: SPDM}) is interchangeably given by
\begin{equation}\label{eq: Mtik}
	M = (H\T H + \lambda T\T T)^{-1} .
\end{equation}
The extension of these operators to 2D images is done naturally by taking differences in both the vertical and horizontal dimensions, and mathematically this can be handled by taking appropriate Kronecker products (see \cite{sanders2020effective} for details).

For deconvolution problems, there is a major computational advantage to writing the operators appearing in (\ref{eq: FD}) as circulant, making these operators convolutional operators. Hence, when $T$ is a $r$th order finite difference operator, then it is also diagonalized by the Fourier transform given by (see \cite{sanders2018multiscale,sanders2020effective} for details) 
\begin{equation}
	T\T T = \F^{-1} D \F,
\end{equation}
where $D$ is a diagonal matrix, and $D_{jj} = \sin^{2r}\left(\frac{ \pi (j-1) }{N} \right)$, for $j=1,2, \dots , N$.

\subsection{Computational Practices}\label{sec: comp}
Many of the formulas forthcoming in this article appear daunting computational tasks at first glance, e.g. the trace term in (\ref{eq: SURE-lin}). For such general large operators, the exact calculation of this trace would require massive matrix-matrix products (and inverses), and would need to instead be approximated using Monte-Carlo methods that are still prohibitively time consuming. However, in the practice of deconvolution problems, FFTs and the Fourier convolutional theorem can be leveraged to their fullest extent. These computational details are briefly discussed here, and it is presumed the reader can generalize these ideas to all other formulas appearing in this article. 

Suppose $A, B \in \R^{N\times N}$ are two arbitrary 2D convolutional operators (hence circulant matrices). Therefore, they have diagonalized representations as
\begin{equation}\label{eq: Ahat}
	A = \F^{-1} \hat A \F  \quad \text{and} \quad B = \F^{-1} \hat B \F,
\end{equation}
where $\F$ is the 2D unitary DFT operator. The matrices $\hat A$ and $\hat B$ are diagonal with entries given by the vectors $\hat a , \hat b \in \C^N$, respectively, defined by
\begin{equation}
	{\hat a} = \sqrt{N} \cdot \F \vec a_1 \quad \text{and} \quad {\hat b} = \sqrt{N} \cdot  \F \vec b_1,
\end{equation}
and $\vec a_1$ and $\vec b_1$ are the first columns of $A$ and $B$. Then it is easy enough to see that
$$
A B = \F^{-1} \hat A \hat B \F ,
$$
hence
\begin{equation}\label{eq: traceExample}
	\trace (A B) = \hat a\T \hat b .
\end{equation}
Observe this last equation appearing in (\ref{eq: traceExample}) is trivial to compute, requiring at most two FFTs to compute $\hat a$ and $\hat b$ and a very cheap dot product.
In a similar fashion, traces involving inverses are also very simple, e.g. 
$$
\trace (A^{-1} B ) = \sum_{j=1}^N \hat b_j / \hat a_j .
$$
Norms and inner products involving $A$ and $B$ can easily be computed using Parceval's theorem. For example, for two vectors $u, v\in \R^N$, with DFTs given by $\F u = \hat u$ and $\F v = \hat v$, then
\begin{equation}
	\begin{split}
		u\T A B v 
		& = u\T \F^{-1} \hat A \hat B\F  v \\
		& = (Fu)^H \hat A \hat B ( \F v)\\
		& = \sum_{j=1}^N  \hat a_j \hat b_j \overline{\hat u_j } \hat v_j .
	\end{split}
\end{equation}
This again requires at most two FFTs and simple Hadamard products and sums. Note however, for all of our algorithms in the forthcoming work, all necessary Fourier transforms may be evaluated prior to the iterative steps, hence the calculations just described simplify to only dot-product type calculations between vectors the same size as the original image.

\section{Fixed Point Methods for PSF Estimation}
The goal proceeding is to make use of SURE to optimize over certain model parameters for the forward operator, $H$. For example, $H$ may be modeled as a Gaussian PSF convolutional operator with an unknown variance. Then we may use SURE to optimize over this free parameter, let's call it $\gamma$, so $H = H(\gamma)$. Then minimizing (\ref{eq: SURE-lin}) over all $\gamma$ leads to the condition
\begin{equation}\label{ptheta}
	0 = \frac{\rmd}{\rmd \gamma} \left( \| H u - b \|_2^2 + 2\sigma^2 \text{trace}(H M H\T ) \right),
\end{equation}
where again $u  = \M  H\T b$.

Two important shorthand notations are introduced here. For any diagonalizable matrix $A \in \R^{N\times N}$, with a diagonalized form given by
\begin{equation}\label{eq: Rdef1}
	A = S^{-1} \Lambda S,	
\end{equation}
we denote
\begin{equation}\label{eq: Rdef2}
	\begin{split}
		A_\Re &:= S^{-1} \Re (\Lambda ) S,\\
		\Rtrace(A )& := \Re \left( \trace(A)\right)= \sum_{j=1}^N \Re\left( \Lambda_{jj} \right) ,
	\end{split}
\end{equation}
where for a complex number $z \in \C$, the operation $\Re (z)$ takes the real part only.
\begin{theorem}\label{prop: main}
	Suppose $H$ is a convolutional operator depending on an arbitrary parameter $\gamma$ with the derivative denoted by $H_\gamma := \frac{\rmd}{\rmd \gamma} H(\gamma)$. Then the value of $\gamma$ which minimizes the SURE criterion in the right-hand side of (\ref{eq: SURE-lin}) satisfies
	\begin{equation}\label{eq: Dgamma}
		\begin{split}
		0 &= 4 \sigma^2 \lambda \cdot {\Rtrace} (\M^2 T\T T H\T H_\gamma ) \\
		&+  4b\T (H \M H\T- I)\T ( HM H_\gamma\T)_\Re  (I - H M H\T)b,
		\end{split}
	\end{equation}
	and hence for any $p\in \R$,
		{\small{
	\begin{equation}\label{eq: gammaF}
		\gamma = \gamma
		\left( \frac{-\sigma^2 \lambda \cdot {\Rtrace} (M^2 T\T T H\T H_\gamma )}
		{b\T (H M H\T- I)\T (H M H_\gamma\T)_\Re (I - H M  H\T)b}\right)^p .
	\end{equation}
}}
\end{theorem}

The proof of Theorem \ref{prop: main} is provided in the appendix. Equation (\ref{eq: gammaF}) is the basis for a fixed point iteration for $\gamma$, namely at the $k+1$ iteration the update on $\gamma$ takes the form
{\footnotesize{
\begin{equation}\label{eq: FPmain}
	\gamma_{k+1}  = \gamma_{k}
	\left( \frac{-\sigma^2 \lambda \cdot {\Rtrace} (M^2 T\T T H\T H_{\gamma_k} )}
	{b\T (H M H\T- I)\T (H M H_{\gamma_k}\T)_\Re  (I - H M  H\T)b}\right)^p	.
\end{equation}
}}
Alternatively, equation (\ref{eq: Dgamma}) may be used as a single-variable gradient descent method for determining $\gamma$, that is
{\small{
\begin{equation}\label{eq: FPGD}
	\begin{split}
	\gamma_{k+1} &= \gamma_k - 4 \tau_k (  \sigma^2 \lambda \cdot {\Rtrace} (\M^2 T\T T H\T H_{\gamma_k} ) \\
	&+  b\T (H \M H\T- I)\T (HM H_{\gamma_k}\T)_\Re (I - H M H\T)b) ,
	\end{split}
\end{equation}
}}
for some $\tau_k>0$.

\subsection{A fixed point method for finding $\lambda$} 
Since we do not have access to the true power spectrum of $u$ given in (\ref{eq: Wie}) as $V$, we use the surrogate approximation of $\lambda T\T T$ from (\ref{eq: Tik}) in place of $\sigma^2 V$ in (\ref{eq: Wie}). To that end, it is also desirable to optimize (\ref{eq: SURE-lin}) over $\lambda$, where $M$ is a function of $\lambda$ given by
\begin{equation}
	M(\lambda ) = (H\T H + \lambda T\T T)^{-1}.
\end{equation}
The following is the main result needed for the fixed point algorithm for estimating $\lambda$ to further optimize SURE.

\begin{theorem}\label{lambda-theorem}
	Consider the Tikhonov regularized solution $u$ given in (\ref{eq: Tik}). Then the optimal $\lambda$ for this solution which minimizes the SURE criterion in (\ref{eq: SURE-lin}) satisfies
	\begin{equation}
		\lambda = \frac{\sigma^2 \trace(H \M T\T T  \M H\T )}
		{b\T H \M T\T T \M T\T T \M H\T b},
	\end{equation}
\end{theorem}
This theorem is also proven in the appendix and is the basis for a fixed point iteration for $\lambda$ given as
\begin{equation}\label{eq: FPlambda}
	\lambda_{k+1} = \frac{\sigma^2 \trace(H \M_k T\T T  \M_k H\T )}
	{b\T H \M_k T\T T \M_k T\T T \M_k H\T b},
\end{equation}
where it is implied that $M_k = M(\lambda_k)$.

\section{Parametric PSFs}
% In this section we describe some PSFs of interest and some of their properties. Of particular importance is that the Fourier transform of these PSFs are real, one of the main features of Theorem \ref{prop: main}. It is worth noting however that similar but slightly more complicated formulas can be derived when the Fourier transforms are not real. Nonetheless, we are primarily interested in those that are real, as we will show below.

In the convention that follows, all images (and PSFs) are considered to be in $\R^{m\times n}$, where $N = m \cdot n$. We also assume $m$ and $n$ are even. When either of the image dimensions are not even, the claims change only very mildly, and the reader can easily modify the equations accordingly. Sticking to this convention keeps the exposition simpler. In the associated algorithms the difference between odd and even dimensions are handled with only a few "if" statements.

In what follows, the indexing for the PSFs is written as
\begin{equation}\label{eq: hxy}
	\begin{split}
		h(x,y) , \quad  x &= -n/2 , -n/2 + 1, \dots , n/2 -1,  \\
		y &= -m/2 , -m/2 + 1 , \dots , m/2 -1.
	\end{split}
\end{equation}
Similarly, the indexing for the discrete Fourier transform (DFT) of the PSF is written as
\begin{equation}
	\begin{split}
		\hat h(k_x , k_y ) , \quad  
		n\cdot k_x &= -n/2 , -n/2 + 1, \dots , n/2 -1,  \\
		m\cdot k_y &= -m/2 , -m/2 + 1 , \dots , m/2 -1,
	\end{split}
\end{equation}
where 
\begin{equation}
	\hat h(k_x ,k_y) = \sum_{x = -n/2}^{n/2 -1 } \sum_{y = -m/2}^{m/2 - 1} h(x,y) e^{-i{2\pi} ( x {k_x} + y {k_y})} .
\end{equation}

\begin{comment}
Considering only 1D signals first, it is easy to deduce that the Fourier transform is real if the signal is \emph{even}. In 2D, it is slightly more involved, and it is characterized in the following proposition.

\begin{prop}\label{realProp}
Consider a 2D image $f \in \R^{m\times n}$ with $m$ and $n$ even.
The imaginary part of $\hat f ( k_x ,k_y) $ vanishes for all $k_x, k_y$ if
\begin{equation}\label{eq: ReCond}
f(x,y) = f(-x , -y)  \quad \text{and} \quad f(-x, y ) = f( x , -y),
\end{equation}
for all combinations of $x,y$ in (\ref{eq: hxy}).
\end{prop}

\end{comment}

In the remainder of this section some useful PSF models with parameters to optimize are described. The PSFs are written with a "proportional to" symbol, $\propto$, to indicate that an additional constant is needed so that they are always normalized for the pixels to sum to one. Since the whole optimization is solved in Fourier domain, in practice it is actually only the DFT of the PSF that is needed. % The DFT of these PSFs are written explicitly here as real valued, and it can also easily be confirmed that they satisfy Proposition \ref{realProp}. 

\subsection{Gaussian PSF Parameters}
Some of the simplest and most useful parametric PSFs come from Gaussian distributions. The \emph{non-angled} form of an anisotropic Gaussian distribution given by
\begin{equation}\label{eq: Gaus1}
	h({x},y; \omega_x , \omega_y ) \propto  \exp \left( -\frac{x^2}{2\omega_x^2} - \frac{y^2}{2\omega_y^2}\right). 
\end{equation}
The standard deviations $\omega_x, \omega_y$ are parameters to solve for using the fixed point (\ref{eq: FPmain}). For the purpose of the optimization, its Fourier transform is given by
\begin{equation}\label{eq: hGaus1}
	\hat h( k_x , k_y ) = \exp \left( -{2 \pi^2} \left[ \left( {\omega_x k_x}\right)^2 + \left({\omega_y k_y}\right)^2 \right] \right),
\end{equation}
which are the eigenvalues of $H$ needed for the fixed point evaluation of (\ref{eq: FPmain}). 

The derivative also needed to evaluate the fixed point (\ref{eq: FPmain}) over $\omega_x$ is
\begin{equation}
	\frac{\partial}{\partial \omega _x} \hat h(k_x , k_y) = -\frac{4\pi^2}{n^2} k_x^2 \omega_x \hat h(k_x,k_y),
\end{equation}
and likewise for $\omega_y$. These are the eigenvalues of the operator $H_\gamma$ appearing in (\ref{eq: FPmain}), where before $\gamma$ was arbitrary, now $\gamma = \omega_x$.

An angled Gaussian PSF with angle $\theta$ is given by
\begin{equation}\label{eq: rot1}
	g(x,y; \omega_x , \omega_y , \theta) = h((x,y) Q_\theta\T ),
\end{equation}
where $h$ is given in (\ref{eq: Gaus1}) and 
\begin{equation}\label{eq: rot2}
	Q_\theta = \left[
	\begin{array}{cc}
		\cos \theta & \sin \theta \\
		-\sin \theta & \cos \theta 
	\end{array} \right].
\end{equation}
Hence one can additionally optimize over the parameter $\theta$. An example of such a PSF is shown in Figure \ref{fig: GausPSF}. The recovered PSF shown in this Figure is the estimated PSF using our fixed point algorithm in the example described later on in Figure \ref{fig: convAll}.

It is straightforward to show that the Fourier transform of this angled Gaussian is given by
\begin{equation}\label{eq: rot3}
	\hat g(k_x, k_y; \omega_x , \omega_y , \theta) = \hat h((k_x, k_y) Q_\theta\T ),
\end{equation}
where $\hat h$ is defined in (\ref{eq: hGaus1}). The derivatives over $\omega_x,  \omega_y$, and $\theta$, which are needed to obtain the eigenvalues of $H_\gamma$, are given by
\begin{equation}\label{eq: Gtheta}
	\begin{split}
		\frac{\partial}{\partial \omega _x} \hat g(k_x , k_y) 
		& = - 4\pi^2  \omega_x \left( {k_x} \cos \theta + {k_y} \sin \theta\right)^2  \hat g(k_x,k_y),\\
		\frac{\partial}{\partial \omega _y} \hat g(k_x , k_y) 
		& = - {4\pi^2} \omega_y \left(- {k_x} \sin \theta +  {k_y} \cos \theta \right)^2  \hat g(k_x,k_y) , \\
		\frac{\partial}{\partial \theta} \hat g(k_x , k_y) 
		& = -4 \pi^2 (\omega_x^2 - \omega_y^2 )\left( {k_x} \cos \theta  +  {k_y}\sin \theta \right)\cdot \\
& \quad		\left( -{kx} \sin \theta + {k_y} \cos \theta\right) \hat g(k_x , k_y).
	\end{split}
\end{equation}

\begin{figure}
	\centering
	\includegraphics[width=.5\textwidth]{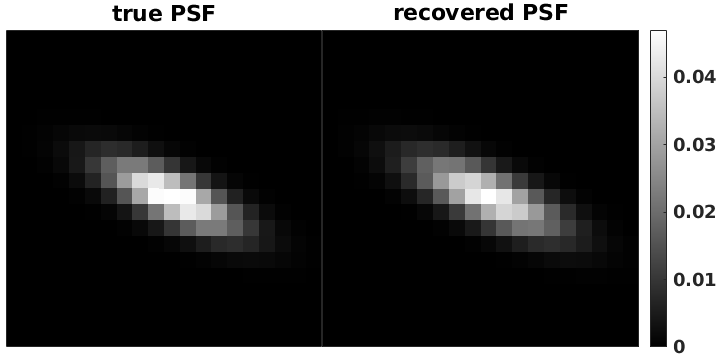}
	\caption{A Gaussian PSF with $\omega_x  = 3$, $\omega_y = 1$, and $\theta=25\degree$. On the left is the true PSF and on the right is the recovered PSF using the fixed point algorithm described in the example in Figure \ref{fig: convAll}.}
	\label{fig: GausPSF}
\end{figure}

\subsection{PSFs Defined in Fourier Domain}
To evaluate the fixed point iterations in (\ref{eq: FPmain}), in practice one only needs the form of the DFT of the PSF instead of the real space domain PSF definition. The following PSF is defined explicitly in Fourier domain:
{\small{
\begin{equation}\label{eq: LapFPSF}
	\hat h(k_x , k_y ; \alpha_x , \alpha_y) =
	\frac{1}{1 +  \alpha_x \sin^2 (\pi k_x)}\,
	\frac{1}{1 +  \alpha_y \sin^2 (\pi k_y)}.
\end{equation}
}}
In the proposition below it is shown that the function defined in (\ref{eq: LapFPSF}) is apparently the DFT of a Laplacian distribution. We feel this result is generally unknown so we include it here, while still emphasizing that other arbitrary models could be designed strictly in Fourier domain.
\begin{prop}\label{prop: Lap}
	Let $h$ be a normalized discrete 1D Laplace distribution given by
	\begin{equation}
		h(x) \propto \exp(-\beta |x|),
    \end{equation}
	for $x = -n/2 , - n/2 +1 , \dots , n/2-1.$
	Then the DFT of $h$ is given by
	\begin{equation}
		\hat h(k) = \frac{1}{1+ 4 \alpha \sin^2(\pi k )} + O(e^{-Bn/2}),
	\end{equation}
	where $\alpha = \frac{\exp({-\beta})}{(1-\exp({-\beta}))^2}$.
\end{prop}
The proof involves a detailed calculation and is given in the appendix. A straightforward extension of the Proposition to 2D tells us that (\ref{eq: LapFPSF}) is essentially a Laplacian PSF with variances related to $\alpha_x$ and $\alpha_y$. In the same fashion as the Gaussian PSF, for more generalized blind deconvolution an angular component can be added to the PSF parameters as in (\ref{eq: rot1}) - (\ref{eq: rot3}) to obtain the rotated version given in Fourier domain as
\begin{equation}\label{eq: LapPSFangle}
	\hat g(k_x , k_y ; \alpha_x , \alpha_y, \theta):= 
	\hat h((k_x, k_y)Q_\theta\T),
\end{equation}
where $h$ was defined in (\ref{eq: LapFPSF}). The derivatives of the PSF are provided in the appendix.

A similar alternative defined directly in the Fourier domain would be 
\begin{equation}
	\hat h (k_x , k_y) = (1 + \alpha_x k_x^2 + \alpha_y k_y^2)^{-1} ,
\end{equation}
along with its rotated variants. The author has not explored this option, but it seems just as reasonable as using a Gaussian or Laplace PSF model.

\subsection{Combinations of PSFs}
Consider a series of PSFs that may be combined linearly to estimate the final PSF, namely
\begin{equation}
	h (x,y; \{ c_j \}_j )= \sum_{j=1}^K c_j h_j(x,y) ,
\end{equation}
where $h_j$ are unique PSFs. The coefficients satisfy $c_j\ge 0$ and $\sum_j c_j = 1$. The free parameters to optimize over are the coefficients $c_j$, and the derivatives needed for the fixed point iterations are trivial. After updating all of the $c_j$ coefficients in each iteration, they may be scaled to ensure they satisfy the normalization constraints. Using this model, the $h_j$ may be viewed as a basis for PSF space. 

% In addition to this, the PSFs $h_j$ may also have free parameters to iterate on. % For example, we may wish to formulate the PSF as a linear combination of the multiple Gaussian PSFs defined in (\ref{eq: rot3}), while also iterating on the free parameters of each respective PSF. 

One very interesting case we have begun to explore using this model is to randomly generate a large number of PSFs, $\{h_j\}_{j=1}^K$, where for example $K=50$ and $h_j$ are each unique Gaussian PSFs as in (\ref{eq: rot1}). The form of this PSF is very general, and the fixed point iterations are used to sort out the weighting factors, $c_j$. The shape of the resulting PSFs by combining many angled Gaussian PSFs can be far more general than simply one parametric angled Gaussian. This approach is not explored in this article, but rather saved for future work due to time and space.

\section{Numerical Results}
The numerical results presented in this section include solving for the \emph{angled} Gaussian and Laplacian PSF parameters. Namely, these are the PSFs defined in (\ref{eq: rot1}) and (\ref{eq: LapFPSF}), and the free parameters to solve for are $\omega_x, \omega_y, \alpha_x, \alpha_y,$ and $\theta$, as well as the regularization parameter $\lambda$. The derivatives with respect to these free parameters needed evaluate the fixed point equations are given explicitly in (\ref{eq: Gtheta}) and (\ref{eq: Ltheta}).

\subsection{Testing Optimization Parameters}\label{sec: params}
To properly solve for the PSF model parameters, first there are very basic optimization parameters to choose within the fixed point iterations. The optimization parameters determined to be most suitable are summarized together in Table 1 for clarity.

\begin{table}[ht]
	\centering
	\caption{Optimization parameters chosen for each deconvolution model parameter.}
	\begin{tabular}{|l|l|}
		\hline 
		$\omega_x,\omega_y$ & Equation (\ref{eq: FPmain}) used with $p=0.25$.\\
		$\alpha_x , \alpha_y$ & Equation (\ref{eq: FPmain}) used with $p=2$.\\
		$\{ c_j\}_j$ & Equation (\ref{eq: FPmain}) used with $p = -1/2$.\\
		$\theta$ & Equation (\ref{eq: FPGD}) used with spectral step for $\tau_k$.\\
		$\lambda$ & Equation (\ref{eq: FPlambda}) used exactly.\\
		\hline
	\end{tabular}
\end{table}

\begin{figure}
	\centering
	\includegraphics[width=0.45\textwidth]{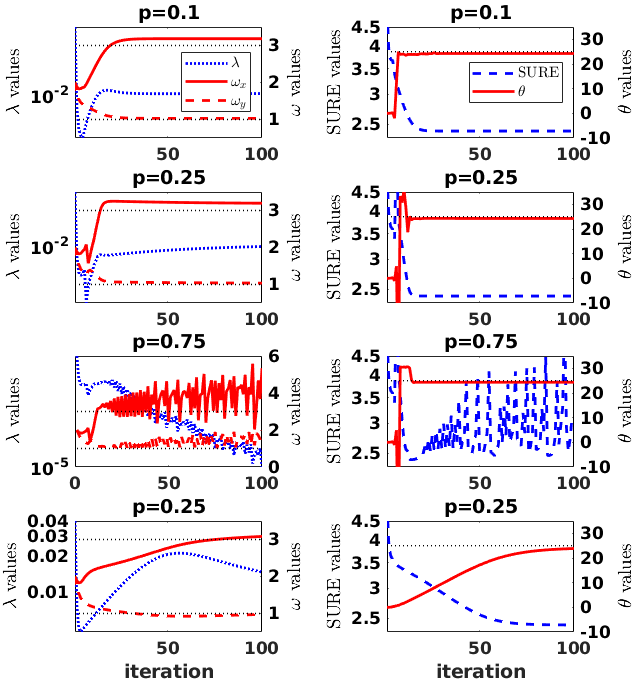}
	\caption{Comparison of the convergence of the fixed point iterations using different optimization parameters and the Gaussian PSF model. The true values of the PSF parameters are $\omega_x = 3,  \omega_y = 1$, and $\theta = 25$, as indicated in the plots by the thin dashed lines.}
	\label{fig: convAll}
\end{figure}

\begin{figure}
	\centering
	\includegraphics[width=0.5\textwidth]{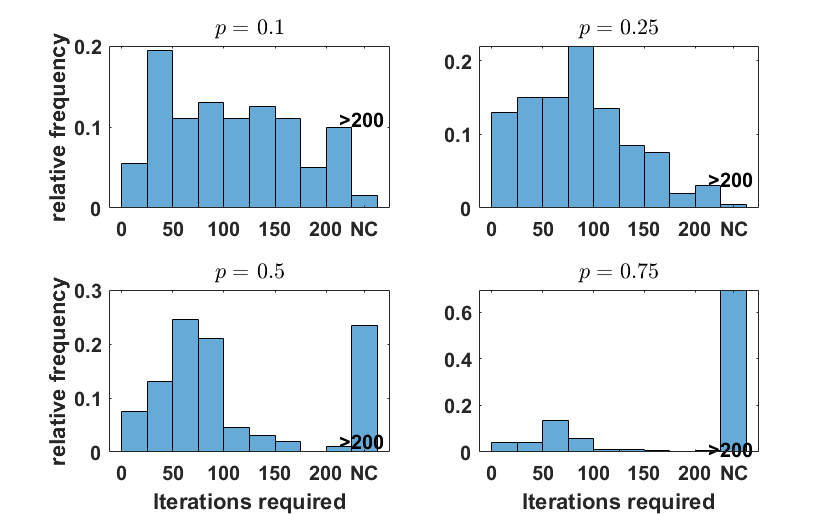}
	\caption{Analysis of the convergence of the fixed point iterations using different optimization parameters for the estimation of the Gaussian PSF parameters. The histogram values indicate the fraction of trials that required the specified number of iterations to converge. The last bar in the histogram indicates the fraction of trials that did not converge (NC), while the second to last histogram bar indicated the fraction that converged but required greater than 200 iterations.}
	\label{fig: histConv}
\end{figure}

In general, for the fixed point iteration in (\ref{eq: FPmain}), using values of $|p|<1$ are more likely to ensure convergence. However, within the range of convergent values of $p$, larger values will lead to faster convergence. Empirically, $p=0.25$ was found suitable for $\omega_x$ and $\omega_y$. For the angle parameter $\theta$, the fixed point equation (\ref{eq: FPmain}) was observed to be unstable for any value of $p$. Instead, the first order gradient descent was used as written in (\ref{eq: FPGD}). The spectral step was used for the step length $\tau$ \cite{barzilai1988two}. Finally, to further improve the convergence, a very small value of $p$ was used in the initial iterations while increasing it incrementally at each iteration until reaching the values just listed. This was not necessary to observe convergence, but rather to improve over-correction in early iterations with poor starting model parameters. The remainder of this section provides a detailed demonstration of the tuning for the Gaussian optimization parameters. The choices for the $\alpha$'s and $\{c_j \}_j$ are listed in Table 1 as a general guideline for the interested reader, based on some manual tuning by the author carried out in a similar fashion to the Gaussian case described below.

First, an example is provided to reveal detailed information within each iteration. This example is shown in Figure \ref{fig: convAll}, where a comparison of the convergence of the fixed point methods for different parameters is shown. The true values for the PSF were $\omega_x = 3$, $\omega_y = 1$, and $\theta = 25\degree$, and this PSF was shown in Figure \ref{fig: GausPSF} along with the recovered estimate from this example. The test image used was the classical \emph{monarch} image used widely in the image processing literature. The noise level, $\sigma$, was set based on a chosen SNR set at 60, where we define the SNR as the mean value of the blurry image divided by $\sigma$.  Each row in the figure shows the convergence of the algorithm for different optimization parameters. The exponent $p$ in (\ref{eq: FPmain}) chosen for each row (going from top to bottom) were $p=$ 0.1, 0.25, 0.75, and 0.25. In the top three rows, the spectral step was used for $\tau$ in the $\theta$ optimization, while the bottom used a fixed step length. Observe that for $p\le 0.25$, all cases converge, while the value $p=0.75$ diverges. The last case that used a fixed step length for $\theta$ shows monotone, albeit slow convergence. The estimation of the optimal $\lambda$ is also provided in these plots, which indicate it settles near $10^{-2}$ in the convergent cases. It could be partly argued that this is the optimal value of $\lambda$, due to the monotone decrease of the SURE objective function at each iteration, which is also shown in the convergence plots. % However, the objective function also includes the PSF parameter updates, so more evidence is necessary. 

Shown in Figure \ref{fig: histConv} is the result from a large set of numerical examples analyzing the convergence of the fixed point iterations for different optimization parameters. In this example, 200 simulated trials were evaluated where the Gaussian PSF parameters and image SNR were randomly generated. The Gaussian standard deviations $\omega_x, \omega_y$ were chosen from a uniform distribution over the interval $[0,5]$, and $\theta$ was similarly chosen uniformly over $[-45\degree , 45\degree]$. For each case, the fixed point algorithms were evaluated using the different $p$ values, and for every example the initial starting values for the algorithm were $\omega_x = \omega_y = 2$ and $\theta = 0\degree$. The histogram values in the Figure indicate the fraction of trials that required the specified number of iterations to converge. The algorithm was considered "converged" when the absolute change in PSF parameter values in multiple subsequent iterations was less than $10^{-3}$. Observe that for $p=0.5$, already almost 25\% of cases do not converge, and when $p=0.75$, the majority of case diverge. On the other hand, comparing the cases of $p=0.1$ and $p=0.25$, nearly all cases converge, but $0.25$ typically required fewer iterations, and many of the cases for $p=0.1$ required greater than 200 iterations. To that end, it is straightforward to conclude that values near $p=0.25$ are optimal.

The convergence of these algorithms could likely be improved with further refinement. One could, for example, implement an accelerated fixed point method \cite{nesterov2003introductory,wang2013scaled}, as was done for the fixed points in \cite{sanders2021real}. Another alternative would be to use additional variable $p$ values for the exponent, based on some simple conditions. An even more detailed approach would be to use a second-order Newton-Raphson method, which requires detailed calculations of all second-order partial derivatives. These detailed calculations are provided in the supplementary material for reference, thought they have not been tested numerically.  Further refinement of the optimization will be considered for future work or for commercial applications.

\subsection{Accuracy of the Gaussian PSF Parameter Estimates}
Here we highlight the accuracy of the parameter values resulting from the simulations in Figure \ref{fig: histConv}, whereas before we were only interested in convergence. The values used are those recovered for the case $p=0.25$, as this was determined to be the most desirable optimization parameter. The scatter plots in Figure \ref{fig: GausScat} compares the true values of these randomly generated parameters with the values recovered using the fixed point algorithm (note that only the first 100 trials are shown to reduce clutter in the scatter plots). Hence, points near the diagonal line $y=x$ indicate accurate results. Notice $\omega_x $ and $\omega_y$ appear fairly accurate across all 100 trials. There is less consistency with the angle, but this can be explained by cases where $\omega_x \approx \omega_y$, in which case the angle is arbitrary. In other words, when $\omega_x=\omega_y$, the PSF is rotationally symmetric. This is indicated in the plot by the color of the dot, where dark dots represent points where $|\omega_x - \omega_y|$ is small. Note that only dark dots appear notably away for the origin. Finally in the bottom left the value of $\lambda$ recovered from our optimization is plotted against a value that was recovered by a brute for search for $\lambda$ and evaluating the SURE objective function each time. Observe that these two values are practically identical, indicating that our algorithm has converged to the optimal $\lambda$. Note here that this fixed point method for $\lambda$ could be used for many other inverse problems.

\begin{figure}
	\centering
	\includegraphics[width=.45\textwidth]{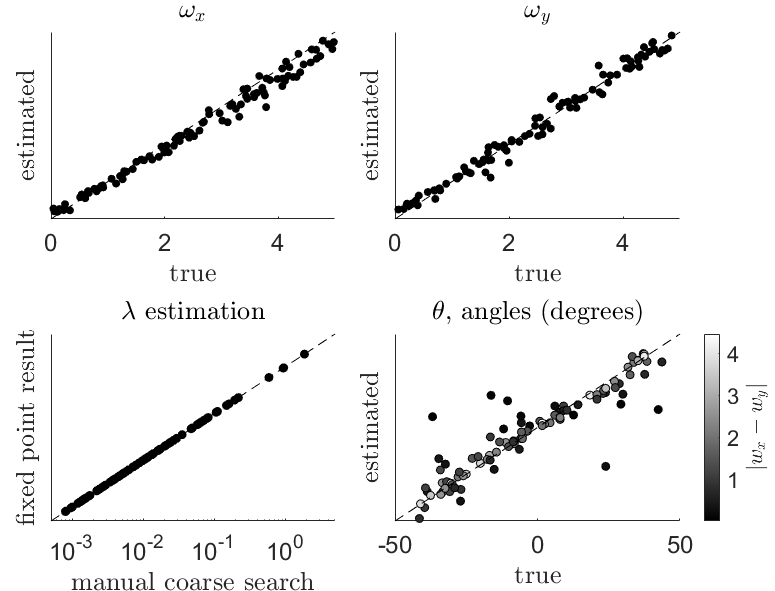}
	\caption{Accuracy of the estimated parameters from the large set of numerical simulations for Gaussian PSFs. Observe that $\omega_x$ and $\omega_y$ are almost always close to their true values, and the angle error is only large whenever the PSF is almost rotationally symmetric ($\omega_x \approx \omega_y$). The bottom left plot shows that $\lambda$ recovered from our algorithm is practically identical to that from using a brute force search.}
	\label{fig: GausScat}
\end{figure}

\subsection{Model Failure for Laplacian PSF Parameter Estimation} \label{sec: LapResults}
Similar to what was done with the Gaussian PSF estimation, the accuracy of the recovered PSF parameters from our algorithm was tested for the PSF form written in (\ref{eq: LapFPSF}), which we argued was essentially a Laplacian PSF. The parameters of interest are $\alpha_x$ and $\alpha_y$, and in each trial they were randomly generated from a uniform distribution over $[0 , 30]$, which roughly corresponded to PSFs with standard deviations within $[0,3.9]$, which can be deduced from Proposition \ref{prop: Lap}. When testing the algorithm, it was discovered that the recovered estimates of these parameters were inaccurate, and typically estimated far too large (see the open red circles in the plots in Figure \ref{fig: LapScatter}). It was determined this was not an issue with the optimization or convergence. In fact, it was observed that the SURE objective function was essentially monotonically decreasing through the iterations, even as the parameters converged to undesirable values. This effects comes from the SURE objective function. Recall that, loosely speaking, Theorem \ref{thmBlu} guarantees that the SURE model will lead us to an ideal PSF whenever estimates for $u$ take the form in (\ref{eq: Wie}), however, with the caveat that the ideal regularization is known (refer to equation (\ref{eq: Wie}) and the proceeding discussion for details).

To demonstrate this concept, the comparison between using the pragmatic Tikhonov regularizer and the ideal regularizer are shown in Figure \ref{fig: LapScatter}. The pragmatic regularization uses the version of $M$ defined in (\ref{eq: Mtik}), while the ideal regularization uses the form of $M$ given in (\ref{eq: SPDM}), where the true $V$ is provided. We are able to form the ideal regularization in these simulations because we have access to the true solution. Observe that the pragmatic Tikhonov regularization results in over-estimation of the parameters in most cases, while the ideal regularization cases are extremely accurate. One intuitive interpretation of this result is that SURE "prefers" a larger blur, so that applying the blur operator to the difference between the true and estimated solutions essentially \emph{wipes out} most of the error. Fortunately this did not occur in the Gaussian PSF case, which strongly indicates it will be a more useful model for real-data cases. One could attempt to ameliorate the Laplacian PSF estimation in several ways. One way, for example, would be a two-step procedure where the estimate for $V $ is provided after an initial image estimate, analogous to that in \cite{dabov2008image}. We have not explored these options and would simply advise one to \emph{stick with the Gaussian.}

\begin{figure}
	\centering
	\includegraphics[width=.45\textwidth]{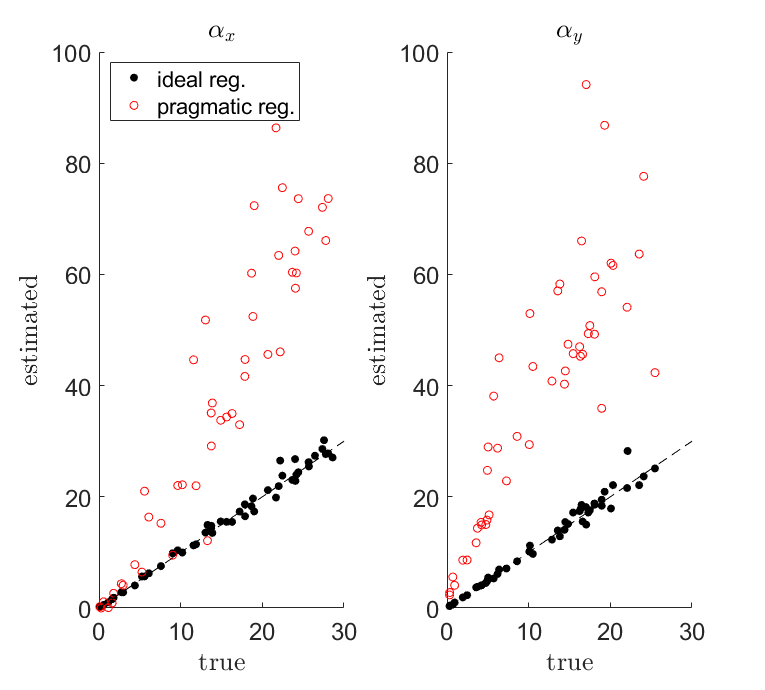}
	\caption{Comparison of recovered estimates of the Laplacian PSF parameters using the ideal regularization versus the pragmatic Tikhonov regularization. The Theorem in \cite{xue2014novel} only guarantees that the SURE criterion recovers the correct PSF when the ideal regularization is known. In this case with the Laplacian PSF, our optimization algorithm still performs well, but the SURE criterion fails when the pragmatic regularization is used.}
	\label{fig: LapScatter}
\end{figure}

\subsection{Real Image Data}
In this section, the numerical approach is demonstrated on an overhead RGB color image containing mild blur and noise using the Gaussian PSF model. The PSF parameters were estimated on the grayscale luminosity channel image, and then each of the RGB channels were deconvolved using the PSF recovered from the luminosity channel. The original image was acquired from a moving aircraft using a rolling shutter CMOS sensor. Since the aircraft was moving in the direction of the rolling shutter (from top of the image to bottom), it was predicted that the primary image blur is in the $y$-axis, hence anticipating that $\omega_y>\omega_x$. Additional inherent blur in the image is assumed resulting from various sources, including camera shake, finite aperture size, imperfect focus, etc. Moreover, the blur within the image is assumed to potentially be spatially variable, due to the rolling shutter capture and variable scene depth.

To deal with the spatially varying blur, a pragmatic approach was taken by first subdividing the image into small overlapping rectangular patches, most of size $256^2$. Each patch was processed and deconvolved independently, and the final image was attained by stitching all of the smaller patches back into a single image. The tiling of the small patches formed a $7\times 8$ tiling over the whole image, hence 56 patches total. The mean ($\mu$) and standard deviations ($\sigma$) of the recovered PSF parameters over all 56 patches is given below:

\begin{equation}
	\begin{split}
		\mu(\omega_x) &= 0.72, \quad \sigma(\omega_x) = 0.10\\
		\mu(\omega_y) & = 0.93, \quad \sigma(\omega_y) = 0.15\\
		\mu(\theta) & = 4.2\degree, \quad \,\,\,\, \sigma(\theta) = 10.3\degree.
	\end{split}
\end{equation}
Observe that the recovered values for $\omega_y$ are significantly larger, as anticipated due the the rolling shutter and aircraft movement. However, it is interesting to note that this additional blur in the $y$-axis cannot be observed visually by the author, though the algorithm can still parse it out. Stable convergence of our algorithm over a single image patch is demonstrated in Figure \ref{fig: realConv}.

\begin{figure}
	\centering
	\includegraphics[width=0.5\textwidth]{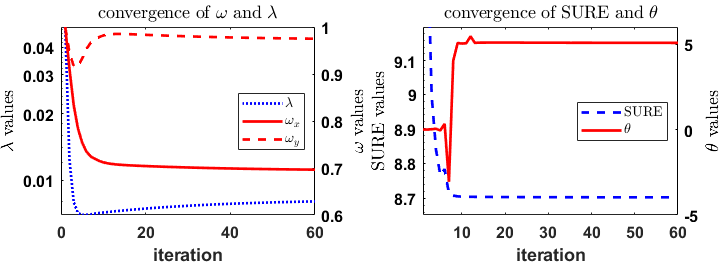}
	\caption{Convergence of the deconvolution parameters over a single image patch for the real data example. Observe that the SURE objective function values monotonically decrease.}
	\label{fig: realConv}
\end{figure}

The accuracy of these parameters is empirically demonstrated via the recovered image quality, which is shown in Figure \ref{fig: real} along with the original images. Our SURE algorithm was only used to quickly estimate the PSF parameters. Then, to obtain a high quality deconvolved image, a BM3D deconvolution approach was taken \cite{dabov2008image}, where the fast GPU-based version of the algorithm was used \cite{sanders2021new}. To further improve the image quality, a super resolution modification of the algorithm was implemented which increases the resolution by a factor of 2 in each dimension. This was useful due to the observed pixelation. The iterative plug and play prior approach with BM3D was also tested \cite{sreehari2016plug, zhang2017learning}, but our empirical tests found no improvement in image quality over the more straightforward BM3D approach, which requires far less computational time. The deconvolved images seen in Figure \ref{fig: real} show substantially improved image quality over the original images. Namely, the pixelation is reduced, image sharpness is improved, and image noise is removed.

\section{Conclusions}
Optimization techniques were proposed and tested for rapid estimation of a large number of PSF parameters with SURE as the underlying objective function. These optimization techniques make it possible to solve for far more PSF parameters using the SURE criterion than what was previously done. It was demonstrated that for properly chosen optimization parameters, the algorithms tend to converge in less than 200 iterations. For Gaussian PSFs, the simulations indicated that the SURE approach combined with the optimization effectively recovered accurate PSF estimates, as well as ideal regularization parameters. This was empirically demonstrated on real image data, where high quality super resolved image was attained using a BM3D-based deconvolution approach with the Gaussian PSF obtained from our SURE algorithm as the input.

For Laplacian PSFs, while the optimization converged to optimal parameters according to the SURE criterion, the parameters were often very inaccurate. This was explained by the lack of access to the ideal regularization for the Wiener filter and highlighted a current shortcoming with the SURE approach for certain PSF models. Other parametric PSF models were suggested in Section 4 that could be tested to extend the application of the optimization techniques beyond what was demonstrated here. Namely, a more general approach was proposed using linear combinations of PSFs as a basis for the final PSF, which leads for far more general PSF shapes.

There is potential room for further improvement and tuning in the optimization. For example, a second order Newton-Raphson method may prove useful, whose necessary ingredients are derived in the appendix.

\begin{figure}[ht!]
	\centering
	\includegraphics[width=0.5\textwidth]{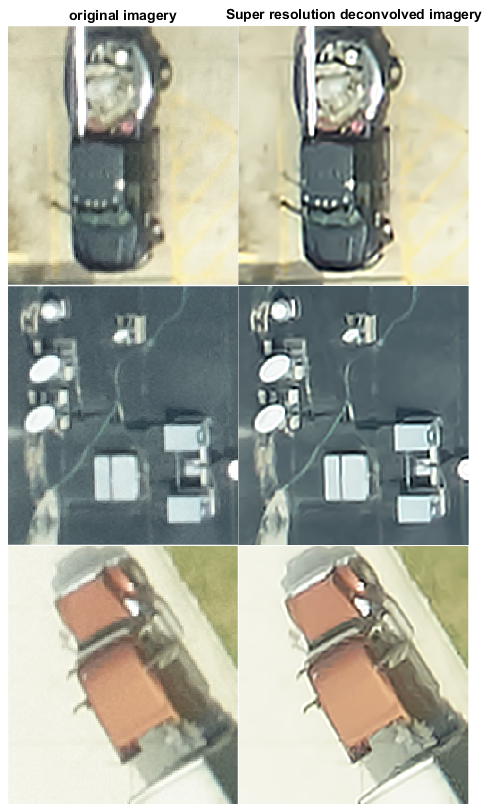}
	\caption{Super resolution deconvolution results on real image data. Three small image patches are compared between the original and the restored image. The top image shows a pickup truck with objects in the truck bed, the middle shows some electrical appliances on top of a building, and the bottom shows a semi-truck. Note that in the restored image, the pixelation is reduced, image sharpness is improved, and image noise is removed. } 
	\label{fig: real}
\end{figure}

\section*{Acknowledgments} 
I want to thank Christian Dwyer, Ulugbek Kamilov, Sean Larkin, and Rodrigo Platte for stimulating and useful discussions relating to this work. I would also like to thank Scott Merritt and the Surdex Corporation for providing the image data. Finally, thanks to the late Robert D. Skeel, whose mathematical techniques inspired this work, and for initially suggesting the idea that is now found in Theorem \ref{lambda-theorem} of this article.

\appendix

\section*{Proof of Theorem \ref{prop: main}}
In this section of the appendix, the main result from Theorem \ref{prop: main} is proven. First Lemmas \ref{Mlemma} and \ref{lem: Rtrace} are needed. The reader should note that the notation introduced in (\ref{eq: Rdef2}) is used extensively below.

\begin{lemma}\label{Mlemma}
	Consider an arbitrary square non-singular matrix $A$ that is dependent upon a parameter $\gamma$.  Then the derivative of the inverse of $A$ with respect to $\gamma$ satisfies
	$$
	\frac{\rmd  }{ \rmd \gamma} A\inv = - A\inv \frac{\rmd A } {\rmd \gamma}  A\inv .
	$$
\end{lemma}
\begin{proof}
	First noting that on one hand
	$$
	\frac{\rmd }{\rmd \gamma} A\inv A = \frac{\rmd}{\rmd \gamma} I = 0,
	$$
	and on the other hand using the product rule
	$$
	\frac{\rmd }{\rmd \gamma} A\inv A = \frac{\rmd A\inv }{\rmd \gamma }A + A\inv \frac{\rmd A }{\rmd \gamma }.
	$$
	Setting the last expression to zero and solving for $\frac{\rmd  }{ \rmd \gamma} A\inv $ completes the proof.
\end{proof}

\begin{lemma}\label{lem: Rtrace}
	Suppose the matrices $A,B,C \in \R^{N\times N}$ are all convolutional operators and that $A$ has real eigenvalues. Then 
	\begin{equation}\label{eq: lem1}
		\trace(AB\T C ) + \trace(A C\T B) = 2\Rtrace(AB\T C )
	\end{equation}
	and 
	\begin{equation}\label{eq: lem2}
		BAC\T + CAB\T = 2(BAC\T)_\Re .
	\end{equation}
\end{lemma}
Equations (\ref{eq: lem1}) and (\ref{eq: lem2}) in Lemma \ref{lem: Rtrace} can be proven in similar fashion to one another. One simple route is to use the concepts about convolutional operators introduced in equations (\ref{eq: Ahat})-(\ref{eq: traceExample}) and simply work through the details. This is left as an exercise to the reader.

\begin{proof}[Proof of Theorem \ref{prop: main}]
	The basis of the proof is to evaluate the derivative of (\ref{eq: SURE-lin}) with respect to $\gamma$, i.e. evaluating (\ref{ptheta}). Using short-hand subscripts to denote derivatives, i.e. $f_\gamma := \tfrac{\rmd}{\rmd \gamma} f $, differentiating the trace term leads to
	{\small{
	\begin{equation}\label{eq: dtrace1}
		\begin{split}
			& \frac{\rmd}{\rmd \gamma } \trace (\M H\T H)\\
			& = \trace(\M_\gamma H\T H ) + \trace(\M H_\gamma\T H) + \trace(\M H\T H_\gamma) \\
			& = \trace(\M_\gamma H\T H ) + 2 \trace_\Re (\M H\T H_\gamma ),
		\end{split}
	\end{equation}
}}
	where the last line follows from Lemma \ref{lem: Rtrace}. According to Lemma \ref{Mlemma}, 
	$$
	\M_\gamma = - \M (H\T H_\gamma + H_\gamma\T H) \M ,
	$$
	hence combining this with Lemma \ref{lem: Rtrace} obtains
	{\small{
	\begin{equation}\label{eq: dtrace2}
		\begin{split}
			\trace(M_\gamma H\T H) 
			& = -\trace(\M (H\T H_\gamma + H_\gamma\T H) \M  H\T H)\\
			& = -2\Rtrace (\M H\T H_\gamma \M  H\T H) .
		\end{split}	
	\end{equation} 
}}
	Putting (\ref{eq: dtrace1}) and (\ref{eq: dtrace2}) together obtains
	\begin{equation}
		\begin{split}
			& \frac{\rmd}{\rmd \gamma } \trace (\M H\T H)\\
			& = 2   \Rtrace(-\M H\T H_\gamma \M H\T H + \M H\T H_\gamma ) \\
			& = 2 \Rtrace(\M H\T H_\gamma (I - \M H\T H )) \\
			& = 2 \Rtrace(\M H\T H_\gamma \M (\M^{-1} - H\T H))\\
			& = 2\lambda  \Rtrace(\M H\T H_\gamma \M T\T T)
		\end{split}
	\end{equation}
	Since every term in the trace above is a convolutional operator, they can be rearranged in any order and the equality still holds. This completes the details of the trace term.
	
	Next observe the normed term can be written as
	$$
	\| (H \M H\T- I)b \|_2^2 = b\T (H\M H\T - I)\T (H\M H\T - I)b,
	$$
	hence 
	\begin{equation}\label{eq: dNorm}
		\frac{\rmd}{\rmd\gamma}  \| H u - b\|_2^2 = 2  b\T (H\M H\T - I)\T (H\M H\T)_\gamma \, b
	\end{equation}
	Next the derivative of $(HMH\T)_\gamma$ needs to be sorted out. This is given by
	\begin{equation}\label{eq: dNorm2}
		\begin{split}
			& (HMH\T)_\gamma \\
			& = H_\gamma M H\T + HMH_\gamma\T + HM_\gamma H\T \\
			& = 2(HMH_\gamma\T)_\Re - H M(H\T H_\gamma + H_\gamma\T H) M H\T \\
			& = 2(HMH_\gamma\T)_\Re - 2H M (H_\gamma\T H)_\R M H\T \\
			& = 2 \left( (HMH_\gamma\T)_\Re - (H M H_\gamma\T)_\Re H M H\T  \right),
		\end{split}
	\end{equation}
	where the last line follows from extensions of Lemma \ref{lem: Rtrace}. The remaining details of the proof are completed by substituting the result from (\ref{eq: dNorm2}) into (\ref{eq: dNorm}), and then combining this with the trace result.
\end{proof}

\section*{Proof of Theorem \ref{lambda-theorem}}

\begin{proof}[proof of Theorem \ref{lambda-theorem}]
	Recall the Wiener filter solution is given by $u^*   = \M H\T b$, where $\M\inv =H\T H + \lambda T\T T$.  Substituting this into the norm in the SURE estimator and expanding leads to
	$$
	\| H u^* -b\|_2^2 = b\T (H \M H\T -I)^2 b
	$$
	Differentiating this with respect to $\lambda$ by using the above expression and Lemma \ref{Mlemma} leads to
	\begin{equation}\label{appendix1}
		\begin{split}
			& \frac{\rmd}{\rmd \lambda} \| H x^* - b \|_2^2 \\
			& = 2b\T (H \M H\T -I) \frac{\rmd}{\rmd \lambda} (H \M  H\T -I)  b\\
			& = 2b\T (I - H \M H\T ) ( H M T\T T M H\T) b \\
			& = 2b\T ( H M T\T T M H\T - H M H\T H M T\T T M H\T)b
		\end{split}
	\end{equation}
	Furthermore, making the substitution $H\T H = M\inv - \lambda T\T T$ further simplifies the expression to
	\begin{equation}\label{appendix3}
		\frac{\rmd}{\rmd \lambda} \| H u^* - b \|_2^2 
		= 2 \lambda b\T H \M T\T T \M T\T T M H\T b
	\end{equation}
	This is the first term needed.  Next, we need the derivative of the last term in (\ref{eq: SURE-lin})
	$$
	\frac{\rmd} {\rmd \lambda} \trace (H \M H\T) 
	$$
	This is straightforward application of Lemma \ref{Mlemma}, since the trace is a linear operator, hence
	\begin{equation}\label{appendix2}
		\begin{split}
			\frac{\rmd} {\rmd \lambda} \trace ( H \M H\T)  
			& = \trace (H  \frac{\rmd} {\rmd \lambda} M H\T ) \\
			& =  - \trace (H M T\T T M H\T ) \
		\end{split}
	\end{equation}
	Combining the results from (\ref{appendix3}) and (\ref{appendix2}) to evaluate the derivative of (\ref{eq: SURE-lin}) leads to
	\begin{equation}
		\begin{split}
			0 &=  2 \lambda b\T H M T\T T M T\T T M H\T b \\
			&\quad - 2\sigma^2  \trace (H \M T\T T M H\T ) 
		\end{split}
	\end{equation}
	Solving for $\lambda$ completes the proof.
\end{proof}

\section*{Second Order Partial Derivatives}
Here the second order derivatives of the SURE objective function are evaluated. These is provided for reference and can be used for a Newton-Raphson optimization scheme. 
To that end, let us differentiate the objective function with respect to two arbitrary PSF parameters, $\rho$ and $\gamma$, which we begin as
\begin{equation}\label{eq: 2nd1}
	\begin{split}
		& \frac{\rmd}{\rmd \rho}\frac{\rmd}{\rmd\gamma}
		\left( \| H u^* - b \|_2^2 + 2\sigma^2 \text{trace}(H M H\T ) \right)
		\\
		& = 	\frac{\rmd }{\rmd \rho} \big[
		4 \sigma^2 \lambda \cdot {\trace} (\M^2 V H\T H_\gamma ) \\
&		+  4b\T (H \M H\T- I)\T HM H_\gamma\T (I - H M H\T)b
		\big]
	\end{split}
\end{equation}
First the trace term in (\ref{eq: 2nd1}) is evaluated:
{\scriptsize{
\begin{equation}
	\begin{split}
		& \frac{\rmd}{\rmd \rho} \trace(M^2 H\T H_\gamma T\T T) \\
		& = \trace((2M M_\rho H\T H_\gamma + M^2 H_\rho\T H_\gamma + M^2 H\T H_{\gamma \rho} ) T\T T )\\
		& = \trace((-4 M^2 H_\rho \T H M H\T H_\gamma + M^2 H_\rho\T H_\gamma  + M^2 H\T H_{\gamma \rho})T\T T)\\
		& = \trace(M^2 T\T T (-4 M H\T H H_\rho\T H_\gamma  + H_\rho\T H_\gamma + H\T H_{\gamma \rho})) .
	\end{split}
\end{equation}
}}

Next the quadratic term is evaluated. Begin by defining $A := (HMH\T - I)$, in which case the quadratic term differentiation simplifies to (ignoring the constant, 4)
\begin{equation}\label{eq: dQ1}
	\begin{split}
		- & \frac{\rmd}{\rmd \rho } b\T A\T H \M H_\gamma\T A y \\
		& = -2 b\T A_\rho\T HM H_\gamma\T A b - b\T A\T \rmd_\rho (HM H_\gamma\T ) A b
	\end{split}
\end{equation}
and $A_\rho$ is evaluated as
\begin{equation}
	\begin{split}
		A_\rho 
		& = H_\rho M H\T + H M_\rho H\T + HM H_\rho\T \\
		& = 2 H_\rho M H\T - 2 H M H\T H_\rho M H\T \\
		& = 		2(I - HMH\T ) H_\rho  M H\T \\
		& = -2 A H_\rho M H\T ,
	\end{split}
\end{equation}
and finally the last term to simplify is
{\small{
\begin{equation}
	\begin{split}
		\rmd_\rho (HM H_\gamma\T ) 
		& = H_\rho M H_\gamma\T + H M_\rho H_\gamma + HM H_{\gamma \rho}\T \\
		& = H_\rho M H_\gamma\T - 2 H M H\T H_\rho M H_\gamma\T + HM H_{\gamma \rho}\T\\
		& =: B(\gamma,\rho) .
	\end{split}
\end{equation}
}}
Hence, (\ref{eq: dQ1}) simplifies to
\begin{equation}\label{eq: dQ2}
	\begin{split}
		- & \frac{\rmd}{\rmd \gamma }b\T A\T H \M H_\gamma\T A b \\
		& = 4 b\T A\T (HM H_\rho\T) (HM H_\gamma\T) A b -b\T A\T  B(\gamma,\rho) A b
	\end{split}
\end{equation}

Putting it all together, the complete 2nd order derivative is given by
{\scriptsize{
\begin{equation}
	\begin{split}
		&	\frac{\rmd}{\rmd \rho}	\frac{\rmd}{\rmd \gamma} \left( \| H u^* - b \|_2^2 + 2\sigma^2 \text{trace}(H M H\T ) \right)
		\\
		& = 4\sigma^2 \lambda\trace\left( M^2 T\T T \left( -4 M H\T H H_\rho\T H_\gamma  + H_\rho\T H_\gamma + H\T H_{\gamma \rho}\right)\right) \\
		& + 16 b\T A\T (HM H_\rho\T) (HM H_\gamma\T) A b - 4 b\T A\T  B(\gamma,\rho) A b.
	\end{split}
\end{equation}
}}

\section*{Optimizing over the regularization}
Suppose the regularization operator $T\T T$ is a function of an arbitrary parameter $r$ that we also want to optimize. For example, when $T\T T$ is the circulant $r$th order finite difference, then
$$
T\T T = \F^{-1} |\hat T|^2 \F,
$$
where $|\hat T|^2$ is diagonal with
$$
|\hat T_{jj}|^2 = \sin^{2r}(\pi (j-1)/N),
$$
hence
$$
\frac{\rmd}{\rmd r} | \hat T_{jj} |^2 = \log\left(( \sin^{2}(\pi (j-1)/N) \right) \cdot \sin^{2r}(\pi (j-1)/N)
$$

Then differentiating the trace term in (\ref{eq: SURE-lin}) with respect to $r$ leads to
\begin{equation}
	\begin{split}
		\frac{\rmd}{\rmd r} \trace (M H\T H )
		& = \trace (M_r H\T H) \\
		& = -\lambda \trace(M (T\T T)_r M H\T H)
	\end{split}
\end{equation}

Differentiating the quadratic term in (\ref{eq: SURE-lin}) and recalling $u^* = MH\T b$, leads to
{\small{
\begin{equation}
	\begin{split}
	      \frac{\rmd}{\rmd r} \| Hu^* - b\|_2^2 
		& = \frac{\rmd}{\rmd r} \| (H M H\T -I)b \|_2^2\\
		& = \frac{\rmd}{\rmd r} b\T (H MH\T -I)^2 b \\
		& = 2 b\T (HMH\T - I) H M_r H\T b \\
		& = -2 \lambda b\T (HMH\T - I) M (T\T T)_r M H\T b,
	\end{split}
\end{equation}
}}
hence the full derivative of (\ref{eq: SURE-lin}) with respect to $r$ is
\begin{equation}
	\begin{split}
		\frac{\rmd}{\rmd r} \text{SURE}(u^*) 
		&= -2\lambda \sigma^2  \trace(M (T\T T)_r M H\T H) \\
		&-2 \lambda b\T (HMH\T - I) M (T\T T)_r M H\T b
	\end{split}
\end{equation}

\section*{Derivatives of the Laplacian PSF}
The derivatives of the angled Laplacian PSF parameters are provided here for reference. This PSF was defined in (\ref{eq: LapPSFangle}). These derivatives are necessary for the fixed point iterations used in the simulations in Section \ref{sec: LapResults}. Defining $(k_x^\theta , k_y^\theta) = (k_x , k_y)Q_\theta\T$, then the derivatives of $\hat g$ with respect to the PSF parameters are
{\footnotesize{
		\begin{equation}\label{eq: Ltheta}
			\begin{split}
				\frac{\partial}{\partial \alpha_x} \hat g (k_x, k_y) 
				& = -\sin^2 (\pi k_x^\theta) \left( 1 +  \alpha_x \sin^2 (\pi k_x^\theta)\right)^{-2} \cdot \\
				& \quad \left(1 +  \alpha_y \sin^2 (\pi k_y^\theta)\right)^{-1}\\
				\frac{\partial}{\partial \alpha_y} \hat g (k_x, k_y) 
				& = -\sin^2 (\pi k_y^\theta) \left( 1 +  \alpha_x \sin^2 (\pi k_x^\theta)\right)^{-1}\cdot \\
				& \quad \left(1 +  \alpha_y \sin^2 (\pi k_y^\theta)\right)^{-2} \\
				\frac{\partial}{\partial \theta} \hat g (k_x, k_y) 
				& = -2 {\pi k_y^\theta}   \alpha_x  \sin(\pi k_x^\theta) \cos(\pi k_x^\theta) \cdot \\
				& \left( 1 +  \alpha_x \sin^2 (\pi k_x^\theta)\right)^{-2} \left(1 +  \alpha_y \sin^2 (\pi k_y^\theta)\right)^{-1} \\
				& + 2 {\pi k_x^\theta}   \alpha_y  \sin(\pi k_y^\theta) \cos(\pi k_y^\theta)\cdot \\
				& \left( 1 +  \alpha_x \sin^2 (\pi k_x^\theta)\right)^{-1} \left(1 +  \alpha_y \sin^2 (\pi k_y^\theta)\right)^{-2}.
			\end{split}
		\end{equation}
}}

\section*{Proof of Proposition \ref{prop: Lap}}
\begin{proof}[proof of Proposition \ref{prop: Lap}]
	Let $f\in \R^{n}$ be given by
	$$
	f(x) = e^{- \beta |x|}, \quad \text{for} \quad x = -n/2, -n/2 + 1 , \dots , n/2-1.
	$$
	Then
	\begin{equation}
		\begin{split}
			\hat f (k ) 
			& = \sum_{x=-n/2}^{n/2-1} e^{-\beta |x|} e^{-i2\pi k x /n}\\
		\end{split}
	\end{equation}
	The sum may be rewritten to go to infinity, while only introducing an error of $O(e^{-\beta n/2})$. This leads to 
	\begin{equation}
		\begin{split}
			& \hat f(k)  + O(e^{-\beta n/2}) \\
			& = \sum_{x=-\infty}^{\infty} e^{-\beta |x| } e^{-i2\pi k x /n} \\
			& = -1 + \sum_{x=0}^\infty e^{-\beta x} \left( e^{-i2\pi k x /n} + e^{i2\pi kx /n} \right) \\
			& = -1 + \frac{1}{1-e^{-(\beta +i2\pi k / n)}} + \frac{1}{1-e^{-(\beta -i2\pi k / n)}}
		\end{split}
	\end{equation}
	To simplify this, first work out the common denominator and simplify:
	\begin{equation}
		\begin{split}
			& (1-e^{-(\beta +i2\pi k / n)})(1-e^{-(\beta -i2\pi k / n)}) \\
			& = 1+ e^{-2 \beta} - e^{-\beta} (e^{-i2\pi k /n} + e^{i2\pi k /n })\\
			& = 1 + e^{-2\beta} - e^{-\beta} \cdot 2 \cdot \cos\left(2 \pi k / n\right) \\
			& = 1 + e^{-2\beta} - 2e^{-\beta} + 2e^{-\beta} - e^{-\beta} \cdot 2 \cdot \cos\left(2 \pi k / n\right) \\
			& = (1-e^{-\beta})^2 + 4e^{-\beta} \sin^2 \left( \pi k / n\right) 
		\end{split}
	\end{equation}
	The numerator is simplified in a similar fashion:
	\begin{equation}
		\begin{split}
			& (1-e^{-(\beta +i2\pi k / n)})+ (1-e^{-(\beta -i2\pi k / n)})\\
			& = 2 - e^{-\beta} (e^{-i2\pi k/n} + e^{i2\pi k/n})\\
			& = 2 - 2e^{-\beta} \cos(2\pi k/n ) \\
			& = 2 - 2e^{-\beta}  + 2e^{-\beta} - 2e^{-\beta} \cos(2\pi k/n )\\
			& = 2(1-e^{-\beta}) + 4 e^{-\beta} \sin^2 (\pi k /n)
		\end{split}
	\end{equation}
	Putting it all together obtains
	{\small{
	\begin{equation}
		\begin{split}
			& \hat f(k) + O(e^{-\beta n/2}) \\
			& = 
			-1 + \frac{2(1-e^{-\beta}) + 4 e^{-\beta} \sin^2 (\pi k /n)}{(1-e^{-1})^2 + 4e^{-\beta} \sin^2 \left( 2 \pi k / n\right)}\\
			& = \frac{2(1-e^{-\beta}) - (1-e^{-\beta})^2 }{(1-e^{-\beta})^2 + 4e^{-\beta} \sin^2 \left(  \pi k / n\right)}\\
			& = \frac{1-e^{-2\beta}  }{(1-e^{-\beta})^2 + 4e^{-\beta} \sin^2 \left(  \pi k / n\right)} \\
			& = C \frac{1}{1 + 4 \alpha \sin^2 \left(  \pi k / n\right) },
		\end{split}
	\end{equation}
}}
	where 
	$$C = \frac{1-e^{-2\beta}}{(1-e^{-\beta})^2}, \quad \alpha = \frac{e^{-\beta}}{(1-e^{-\beta})^2}.$$
	In the proposition, $h$ is the normalized version of $f$, which implies $\hat h(0) = 1$. Substituting this condition in completes the proof.
\end{proof}

% \bibliography{my_bib}
% \bibliographystyle{abbrv}

\end{document}